\documentclass[11pt]{article}

\usepackage[pdftex]{graphicx}

\usepackage{amssymb}
\usepackage{amsthm}

\usepackage{accents}

\usepackage{tipa}

\makeatletter
\renewcommand\section{\@startsection{section}{1}{\z@}%
                                  {-3.5ex \@plus -1ex \@minus -.2ex}%
                                  {2.3ex \@plus.2ex}%
                                  {\normalfont\large\bfseries}}
\makeatother

\DeclareGraphicsExtensions{.jpg}
\begin{document}

\title{Domination structure for number three}

\author{Misa Nakanishi \thanks{E-mail address : nakanishi@2004.jukuin.keio.ac.jp}}
\date{}
\maketitle

\begin{abstract}
From a research of several
recent papers, in the first part, we are concerned with domination
number in cubic graphs and give a sufficient condition of Reed's conjecture.
In the second part, from a perspective, we study the structure of a minimum dominating
set in 3-connected graphs. It is derived from a collection of cycles with
length 0 mod 3. \\
keywords: cubic graph, 3-connected graph, minimum dominating set, ear decomposition
\end{abstract}

\newtheorem{thm}{Theorem}[section]
\newtheorem{lem}{Lemma}[section]
\newtheorem{prop}{Proposition}[section]
\newtheorem{cor}{Corollary}[section]
\newtheorem{rem}{Remark}[section]
\newtheorem{conj}{Conjecture}[section]
\newtheorem{claim}{Claim}[section]
\newtheorem{fact}{Fact}[section]

\newtheorem{defn}{Definition}[section]
\newtheorem{propa}{Proposition}
\renewcommand{\thepropa}{\Alph{propa}}
\newtheorem{conja}[propa]{Conjecture}

\section{Introduction}
\label{intro}

\noindent In this paper, a graph $G$ is simple and undirected with a vertex set $V$ and an edge set $E$. We follow \cite{Diestel}, \cite{Ota} for notations and properties. For $v \in V$, the open neighborhood is denoted by $N_G(v)$ and the closed neighborhood is denoted by $N_G[v]$ (or, $N(v)$ and $N[v]$ respectively), also for $W \subseteq V$, $N_G(W) = {\displaystyle \bigcup_{v \in W}} N_G(v)$ and $N_G[W] = {\displaystyle \bigcup_{v \in W}} N_G[v]$ (or, $N(W)$ and $N[W]$ respectively). A {\it dominating
set} $X \subseteq V$ is such that every vertex of $V \setminus X$ is adjacent to a vertex of $X$. A minimum dominating set is called a {\it d-set}. The minimum cardinality taken over all minimal dominating sets of $G$ is the {\it domination number} denoted by $\gamma(G)$. The minimum cardinality taken over all maximal independent sets of $G$ is the {\it independent domination number} denoted by $i(G)$. For a d-set $X$ and a set $R$, $X(R)$ denotes $X \cap R$. For a set of graphs $S$ and $m \geq 1$, $mS$ denotes the vertex-disjoint union of $m$ members. The vertex-disjoint union of $S$ is denoted by $\bigoplus S$.    \\

\noindent For the domination number of a graph, in decades the research on cubic graphs has intensively studied that show several important results. The complexity of minimum dominating set (MDS) in cubic graphs is NP-hard \cite{Alimonti}. A random 3-regular graph asymptotically almost surely has no 3-star factors \cite{Assiyatun}. Reed indicated that almost all cubic graphs are Hamiltonian, also the upper bound of the domination number of a connected cubic graph $G$ is conjectured as $\lceil |G|/3 \rceil$ \cite{Reed}. Then the counterexamples that exceed the bound have shown, for example, there is an extremal graph of the domination number 21 over 60 vertices following the series of cubic graphs beyond the boundary \cite{Kostochka} \cite{Kelmans}. \\

\noindent In the first part, we show that the connected cubic graphs that have the domination number above the bound have a minimum dominating set as an independent set. Otherwise, the conjecture is true. \\

\noindent A sufficient condition for $\gamma(G) = i(G)$ was represented for a general graph $G$ as an induced subgraph isomorphic to $K_{1, 3}$, also called 3-star, free \cite{Allan}. Then an induced subgraph, say $I$, is defined as $N[v_1] \cup N[v_2]$ on two adjacent vertices $v_1$ and $v_2$ with degree at least three we call core. We observe $I$ as a forbidden subgraph for $\gamma(G) = i(G)$ with the simplest proof.

\begin{propa}[\cite{Allan}]
If $G$ does not have an induced subgraph isomorphic to $K_{1, 3}$, then $\gamma(G) = i(G)$.
\end{propa}

\begin{propa}[\cite{Cockayne}]
For a graph $G$, if $I \not \subseteq G$ then $\gamma(G) = i(G)$. 
\end{propa}

\begin{proof}
Let $X$ be a d-set of $G$ and $E(X)$ be minimal. For $x, y \in X$ such that $xy \in E(G)$, let $d_{G}(x) = 2$ and $N_{G}(x) \setminus \{y\} = \{x'\}$. For all $z \in N_{G}(x') \setminus \{x\}$, if $z \notin X$ then $(X \setminus \{x\}) \cup \{x'\} = X'$ that is a d-set. $||X|| - 1 \geq ||X'||$ contrary to the minimality of $E(X)$. If there is $z \in N_{G}(x') \setminus \{x\}$ such that $z \in X$ then $X \setminus \{x\} = X''$ that is a d-set and contrary to the minimality of $V(X)$. 
\end{proof}

\noindent A 3-connected cubic graph was conjectured as the difference between the independent domination number and the domination number is one, but it was disproved as taken in infinity.

\begin{propa}[\cite{Zverovich}]
For any $c \in \{0, 1, 2, 3\}$ and any integer $k \geq 0$ there exist infinitely many cubic graphs with connectivity $c$ (say one as $G$) for which $i(G) - \gamma(G) = k$.
\end{propa}

\noindent The next statement is suggested and has been widely discussed. 

\begin{conja}[\cite{Reed}]
Every connected cubic graph $G$ contains a dominating set of at most $\lceil |G|/3 \rceil$ vertices.
\end{conja}

\noindent It has some counterexamples in cubic graphs with connectivity one and two. A graph $H_4$ in \cite{Kostochka} is observed as it has a minimum dominating set as an independent set. \\

\noindent In the second part, we refine the structure of path covers used in the proof of cubic graph domination by Reed \cite{Reed}. According to the proof, a cubic graph is almost covered by paths of length 0 mod 3. By finding cycles of length 0 mod 3 in a 3-connected graph, we reveal the structure of 3-connected graphs. As for a 2-connected graph, its structure is common to us.  

\begin{propa}[\cite{Diestel}]
A graph is 2-connected if and only if it can be constructed from a cycle by successively adding $H$-paths to graphs $H$ already constructed.
\end{propa}

\noindent Let $C_G$ be a collection of cycles with length 0 mod 3 for a graph $G$.
Two cycles are connecting without {\it seam} if and only if one of them is constructed
by adding one ear (not a cycle) to the other. Let $C_{SG}$ be a maximal
subset of $C_G$ such that cycles are connecting without seam. Let $D_{SG}$ be a maximal subset of $C_{SG}$ such that cycles are connecting without seam but dropped when no exclusive vertex is contained. Let $\texttoptiebar{C}_{SG}$ be a graph composed of members in $C_{SG}$. {\it $X$-3-paths are assigned} to $C_{SG}$ or $D_{SG}$ if and only if every cycle of $C_{SG}$ or $D_{SG}$ has a vertex of $X \subseteq V$ at every three vertices. \\

\section{A sufficient condition of Reed's conjecture}
\label{sec:3}

\begin{thm}\label{T1}
For a connected cubic graph $G$, if $\gamma(G) > \lceil |V|/3 \rceil$ then $\gamma(G) = i(G)$.
\end{thm} 
 
It is central for this proof how the edges and vertices are deleted from a cubic graph to preserve its dominating set. By deleting a vertex of $G$, some path is broken.
A substitution is needed to connect them and preserve an original set of vertices of the dominating set. 

\begin{lem}\label{disjoint}
For a graph $G$ with $\Delta(G) \leq 3$ and a d-set $X$ with $E(X)$ minimal and nonempty, if $v_1, v_2, w \in X$ and $v_1v_2 \in E(X)$ then $N_G[\{v_1, v_2\}] \cap N_G[w] = \emptyset$.
\end{lem}

\begin{proof}
For $x_1, x_2 \in V(X)$, let $x_1x_2 \in E(X)$. Suppose $N(x_1) \setminus \{x_2\} = \emptyset$ then $X \setminus \{x_1\}$ is a d-set, contrary to the minimality of $V(X)$. For $N(x_1) \setminus \{x_2\} = \{v_1\}$, let $v_2 \in N(v_1) \setminus \{x_1\}$. Suppose $v_1 \in X$ or $v_2 \in X$ then $X \setminus \{x_1\}$ is a d-set, contrary to the minimality of $V(X)$. Otherwise, $(X \setminus \{x_1\}) \cup \{v_1\}$ is a d-set, contrary to the minimality of $E(X)$. For $N(x_1) \setminus \{x_2\} = \{v_1, w_1\}$, let $v_2 \in N(v_1) \setminus \{x_1\}$ and $w_2 \in N(w_1) \setminus \{x_1\}$. Suppose $v_1 \in X$ or $v_2 \in X$ then $w_2 \notin X$ by the minimality of $V(X)$. $(X \setminus \{x_1\}) \cup \{w_1\}$ is a d-set, contrary to the minimality of $E(X)$. Thus, $v_1, v_2 \notin X$.
\end{proof}

\begin{defn}
For a graph $G$, $v_1, v_2 \in V$ and $X \subseteq V$, if $v_1v_2 \in E$ is \\
(i) $v_1, v_2 \notin X$ or \\
(ii) $v_1 \in X$, $v_2 \notin X$, $(N(v_2) \setminus \{v_1\}) \cap X \ne \emptyset$ or \\
(iii) $v_1, v_2 \in X$ \\
then a set of $v_1v_2$ is denoted by $U_G(X)$, or $U(X)$.
\end{defn}

\begin{fact}\label{U'}
For a graph $G$, a d-set $X$ and $U' \subseteq U(X)$, $X$ is a d-set of $G - U'$. 
\end{fact}

\begin{defn}
For a graph $G$ and $Y \subseteq V$, if $t_1 \in Y$ has a set of vertices $B(t_1)$ such that for all $b \in B(t_1)$, $b \in N(t_1)$ and $(N[b] \setminus \{t_1\}) \cap Y = \emptyset$, take $B(t_1)$ maximal. $\bigcup_{t_1 \in Y}B(t_1)$ is denoted by $T_G(Y)$, or $T(Y)$. 
\end{defn}

\begin{defn}
For a graph $G$, $Y \subseteq V$ and $v_1 \in T(Y)$, let $t_1 \in N(v_1) \cap Y$. Delete the edge $v_1t_1$. For each $t_2 \in N(v_1) \setminus \{t_1\}$, subdivide the edge $v_1t_2$ by a new vertex $w_2$ respectively. For a set of $v_1$ applied this replacement, say $S$, the resulting graph is denoted by $G(S)$.
\end{defn}

\begin{fact}\label{T'}
For a graph $G$, $Y \subseteq V$ and $T' \subseteq T(Y)$, if $Y$ is a d-set of $G - T'$ then $Y \cup T'$ is a d-set of $G(T')$. 
\end{fact}

\noindent {\it Proof of Theorem \ref{T1}.}
Let $G$ be a connected cubic graph and $X$ be a d-set. Let $E(X)$ be minimal. Let $||X|| > 0$. For $T' \subseteq T(X)$, let $X'$ be a d-set of $G - T'$. We construct a graph $G(T')$ where the sequences of two close vertices of $X$ ($P = a_1a_2$ such that $a_1, a_2 \in X$ or $Q = b_1b_2b_3$ such that $b_1, b_3 \in X$) are on the same path if possible and $|X| - |X'|$ is maximum (as condition (0)). We show by induction for $|X| - |X'|$. First let $|X| - |X'| = 0$. By $||X|| \geq 1$ there exists a path $P = a_1a_2$ and $a_1, a_2 \in X$, which is $P \subseteq G$. Suppose there exists a path $Q = Q_1$ or $Q_2$, where $Q_1 = b_1b_2b_3$ and $b_1, b_3 \in X$, and $Q_2 = c_1c_2$ and $c_1, c_2 \in X$, which is $Q \subseteq G$. Now we construct a graph $G''$ from $G$. For $T' \subseteq T(X)$, let $G(T') = G'$. By Fact \ref{T'}, $X \cup T' = Y$ is a d-set of $G'$. For $U' \subseteq U_{G'}(Y)$, let $G' - U' = G''$. By Fact \ref{U'}, $Y$ is a d-set of $G''$. Suppose there exists $R = P \cdots Q$ in $G'$. Suppose also $R \subseteq G''$ as a path component. Even for a path $R' = \alpha R \beta$ such that $\alpha, \beta \notin Y$, $Y(R')$ is not a d-set of $R'$, also for $R$, a contradiction. So $R \not \subseteq G'$. Since $G$ is connected, $Q_2 \not \subseteq G$. If $Q_1 \subseteq G$ then there exist $Y$-3 paths $P \cdots b_2b_1$ and $P \cdots b_2b_3$ in $G$. By Lemma \ref{disjoint}, $|N_G(b_1) \setminus N_G[X \setminus \{b_1\}]| = 2$ and $|N_G(b_3) \setminus N_G[X \setminus \{b_3\}]| = 2$, so $|X| \leq \lceil |V|/3 \rceil$. Suppose $|X| - |X'| \leq k$ and $|X| \leq \lceil |V|/3 \rceil$. Let $|X| - |X'| = k + 1$. For $a_1, a_2 \in X$, let $a_1a_2 \in E(X)$. For $v \in X \setminus \{a_1, a_2\}$, $v \in S \subseteq N[v]$ with $|S| \geq 3$ and $m \geq 1$, let $G - mS = H$. Let $Z$ be a d-set of $H$ and for $T'_{H} \subseteq T_{H}(Z)$, let $Z'$ be a d-set of $H - T'_{H}$. We construct a graph $H(T'_H) = H'$ which satisfies condition (0) for the corresponding sets. That is, $H'$ contains a path component $R$ that makes $|Z| - |Z'|$. It is not difficult to verify that $||Z|| \geq 1$, $|Z| = |X| - m$ and $|Z| - |Z'| \leq k$ by so taking $mS$ about $R$. By induction, $|Z| \leq \lceil (|V| - m|S|)/3 \rceil$, which implies $3(|X| - m) \leq |V| - 3m + 2$. Thus, $|X| = \lceil (3|X| - 2)/3 \rceil \leq \lceil |V|/3 \rceil$. This completes the proof of Theorem \ref{T1}. \qed

\section{Domination structure in 3-connected graphs}
\label{sec:3} 

\begin{thm}\label{T2}
For a 3-connected graph $G$, ${D}_{SG}$ corresponds to its d-sets.
\end{thm}

\noindent {\it Proof of Theorem \ref{T2}.} 
Let $G$ be a 3-connected graph. Let $R$ be a component of $G - \bigcup \texttoptiebar{C}_{SG}$. First, we assume that $C_{SG}$ is unique, otherwise consider one by one.
 Suppose that $R \ne \emptyset$. Note that $N_G(R) \cap \texttoptiebar{C}_{SG}$ has at least 3 vertices since $G$ is 3-connected. Let $U = N_G(R) \cap \texttoptiebar{C}_{SG}$. Let $t_0, u_0, v_0 \in R$ be adjacent to $t, u, v \in U$ respectively.

\begin{claim}\label{R}
There exist a path from $t_0$ to $u_0$ and a path from $t_0$ to $v_0$ in $R$ which are internally disjoint. 
\end{claim}

\begin{proof} 
It is assumed that $t_0, u_0$ and $v_0$ are distinct. Let $t_1, t_2 \in N(t_0) \cap R$. Let $P$ be a path between $t_1$ and $u_0$ in $R$ and $Q$ be a path between $t_2$ and $v_0$ in $R$. Suppose that $P \cap Q \ne \emptyset$. Let $x$ be the first vertex of $P \cap Q$ from $t_0$. Since $G$ is 3-connected, $G - x - t_0$ is connected. Without loss of generality, there exists a path from a vertex of $P\ring{x}$ to a vertex of $U$ which avoids $x$. 
\end{proof}

\begin{claim}\label{emptyset}
$C_{SG} \ne \emptyset$.
\end{claim}

\begin{proof}
Take any vertex $x \in V$. Since $G$ is 3-connected, $x$ is adjacent to at least 3 vertices. Let $t_0, u_0, v_0 \in N_G(x)$. By Claim \ref{R}, there exist a path $P_1$ from $t_0$ to $u_0$, a path $P_2$ from $u_0$ to $v_0$ and a path $P_3$ from $v_0$ to $t_0$ in $G - x$ which are internally disjoint. It suffices to show that there exists a path $Q$ of length 2 mod 3 with two ends of $\{t_0, u_0, v_0\}$. If $|P_1|$, $|P_2|$ or $|P_3|$ is of length 2 mod 3 then the proof is complete. Let $|P_1| \equiv 0$ and $|P_2| \equiv 0$ mod 3. Then $|P_1P_2| \equiv 2$ mod 3. Otherwise, without loss of generality, two cases arise. 
\begin{itemize}
\item[(I)] $|P_1| \equiv 0$, $|P_2| \equiv 1$ and $|P_3| \equiv 1$ mod 3. 
\item[(II)] $|P_1| \equiv 1$, $|P_2| \equiv 1$ and $|P_3| \equiv 1$ mod 3. 
\end{itemize}
Note that the inner vertices of $P_1$, $P_2$ and $P_3$ have paths between them, otherwise adjacent to $x$. Let $\mathcal{S}_i$ be a set of paths which have an inner vertex $i$ as an end ($1 \leq i \leq p$). For the cases (I) and (II), the followings must be confirmed. \\
Step 1. Confirm that we obtain $Q$ which contains the path $S \in \mathcal{S}_1$. \\
Step 2. If all $S$ have $Q$, stop the steps. Otherwise let $\mathcal{T}_1$ be a set of $S$ which does not have $Q$, and $\emptyset$. \\
Step 3. For $\mathcal{S}_2, \cdots, \mathcal{S}_p$, apply Step 1 and Step 2. \\
Step 4. Confirm that we obtain $Q$ which contains the paths in $T \in \mathcal{T}_1 \times \mathcal{T}_2 \times \cdots \times \mathcal{T}_p$. \\
Step 5. If all $T$ have $Q$, stop the steps. \\
By these steps, we obtain $Q$. 
\end{proof}

\begin{claim}\label{closed path}
If $Y$-3-paths are assigned to $C_{SG}$, any two vertices of $\texttoptiebar{C}_{SG}$ have at least two $Y$-3-paths between them which have distinct penultimate vertices from both ends. 
\end{claim}

\begin{proof}
For two cycles of $C_{SG}$ which connect without seam, say $C_1$ and $C_2$, let $C_2$ be obtained from adding the ear $v_1C_2v_2 = P_1$ to $C_1$. Let $x_1 \in V(C_1)$ and $x_2 \in V(P_1)$. There exist the $Y$-3-paths $x_1C_1v_1C_2x_2$ and $x_1C_1v_2C_2x_2$. Let $x_1 \in V(C_1)$ and $x_2 \in V(C_2 - P_1)$. There exist the internally disjoint $Y$-3-paths $x_1C_1x_2$. Let two cycles $C_k$ and $C_{k + 1}$ of $C_{SG}$ connect without seam ($k \geq 1$). Applying the same argument, $x_k \in V(C_k)$ and $x_{k + 1} \in V(C_{k + 1})$ have at least two $Y$-3-paths between them. It is not difficult to verify that any two vertices $x_i, x_j \in \texttoptiebar{C}_{SG}$ are contained in $C_i$ and $C_j$ respectively, which are in the sequence of cycles connecting without seam in $C_{SG}$. We obtained the claim.
\end{proof}

Let $X$-3-paths be assigned to $C_{SG}$. The vertex $u \in \texttoptiebar{C}_{SG}$ is two types, $u \in X$ or $u \in V \setminus X$. 
If $r \in R$ is such that $N_G(r) \subseteq \{s\} \cup U$ for some $s \in R$ then let $R'$ be a set of $r$. Let $O = \texttoptiebar{C}_{SG} \cup R'$. Let $M$ be a component of $G - O$. Now, let $t_0, u_0, v_0 \in N_G(O) \cap M$. A cycle obtained from a path from $t_0$ to $u_0$, a path from $u_0$ to $v_0$ and a path from $v_0$ to $t_0$ in $M$ which are internally disjoint is denoted by $\blacktriangle$. Let $t, u, v \in O$ be adjacent to $t_0$, $u_0$ and $v_0$ respectively. The vertex $o \in \{t, u, v\}$ is four types ($*1$). 
\begin{itemize}
\item[(a)] $o \in \texttoptiebar{C}_{SG}$ and $o \in X$ 
\item[(b)] $o \in R'$ and $(N_O(o) \cap \texttoptiebar{C}_{SG}) \cap X = \emptyset$ 
\item[(c)] $o \in R'$ and $(N_O(o) \cap \texttoptiebar{C}_{SG}) \cap X \ne \emptyset$ 
\item[(d)] $o \in \texttoptiebar{C}_{SG}$ and $o \in V \setminus X$ 
\end{itemize}
Note that for (c), $(N_O(o) \cap \texttoptiebar{C}_{SG}) \subseteq X$, otherwise $o \in \texttoptiebar{C}_{SG}$. 
If there exists a path $Q$ between two vertices of $\{t, u, v\}$ (through the vertices in $M$) of specific length and types then $C_{SG}$ is not maximal. Twenty cases arise ($*2$) by simple case analysis. For example, a path $Q$ of length 2 mod 3 between types (a) and (a) is. \\

\begin{table}[h]
  \begin{minipage}[t]{.45\textwidth}
    \begin{center}
      \begin{tabular}{|c|c|c|} \hline
    \ & length & types \\ \hline
(1) & 2 mod 3 & (a) and (a) \\ \hline
(2) & 0 mod 3 & (a) and (b) \\ \hline
(3) & 2 mod 3 & (a) and (b) \\ \hline
(4) & 1 mod 3 & (a) and (c) \\ \hline
(5) & 0 mod 3 & (a) and (d) \\ \hline
(6) & 1 mod 3 & (a) and (d) \\ \hline
(7) & 1 mod 3 & (b) and (b) \\ \hline
(8) & 2 mod 3 & (b) and (b) \\ \hline
(9) & 0 mod 3 & (b) and (b) \\ \hline
(10) & 1 mod 3 & (b) and (c) \\ \hline

      \end{tabular}
    \end{center}

  \end{minipage}
  \hfill
  \begin{minipage}[t]{.45\textwidth}
    \begin{center}
      \begin{tabular}{|c|c|c|} \hline
    \ & length & types \\ \hline
(11) & 2 mod 3 & (b) and (c) \\ \hline
(12) & 0 mod 3 & (b) and (d) \\ \hline
(13) & 2 mod 3 & (b) and (d) \\ \hline
(14) & 1 mod 3 & (b) and (d) \\ \hline
(15) & 0 mod 3 & (c) and (c) \\ \hline
(16) & 0 mod 3 & (c) and (d) \\ \hline
(17) & 2 mod 3 & (c) and (d) \\ \hline
(18) & 2 mod 3 & (d) and (d) \\ \hline
(19) & 0 mod 3 & (d) and (d) \\ \hline
(20) & 1 mod 3 & (d) and (d) \\ \hline

      \end{tabular}
    \end{center}

  \end{minipage}
\end{table}

\noindent Note that, by Claim \ref{closed path}, two vertices of (a), (b), (c) and (d) have a $X$-3-path in $\texttoptiebar{C}_{SG}$ but the same pair may have a different connection, which appears in ($*2$). Also, that all of $\{t, u, v\}$ are type (c) is exceptional. Because we have a desired path $Q$ or $M$ itself forms another $C_{SG}$. Note that disjoint sets of $\texttoptiebar{C}_{SG}$ have common neighbors $x \in V(G) \setminus \bigoplus \texttoptiebar{C}_{SG}$ such that $N_G(x) \subseteq X$.

\begin{claim}\label{triangle}
If $M$ has $\blacktriangle$ then $M = \emptyset$.
\end{claim}

\begin{proof}
Suppose that $M \ne \emptyset$. It suffices to show that there exists $Q$ as in ($*2$). Let $P_1$ be a path from $t_0$ to $u_0$, $P_2$ be a path from $u_0$ to $v_0$ and $P_3$ be a path from $v_0$ to $t_0$ in $M$ which are internally disjoint. According to the types ($*1$) of $t, u$ and $v$, if $|P_1|$, $|P_2|$ and $|P_3|$ are specified as ($*2$) then we obtain $Q$. Suppose otherwise. Note that the inner vertices of $P_1$, $P_2$ and $P_3$ have paths between them, otherwise adjacent to $O$. Let $\mathcal{S}_i$ be a set of paths which have an inner vertex $i$ as an end ($1 \leq i \leq p$). Then the followings must be confirmed. \\
Step 1. Confirm that we obtain $Q$ which contains the path $S \in \mathcal{S}_1$. \\
Step 2. If all $S$ have $Q$, stop the steps. Otherwise let $\mathcal{T}_1$ be a set of $S$ which does not have $Q$, and $\emptyset$. \\
Step 3. For $\mathcal{S}_2, \cdots, \mathcal{S}_p$, apply Step 1 and Step 2. \\
Step 4. Confirm that we obtain $Q$ which contains the paths in $T \in \mathcal{T}_1 \times \mathcal{T}_2 \times \cdots \times \mathcal{T}_p$. \\
Step 5. If all $T$ have $Q$, stop the steps. \\
By these steps, we obtain $Q$.
\end{proof}

\begin{claim}\label{R1}
$|R| \leq 1$.
\end{claim}

\begin{proof}
By Claim \ref{triangle}, it suffices to consider that $|R| \leq 2$. Suppose that $R = \{u, v\}$. From the definition of $R$, $u, v \not \in \texttoptiebar{C}_{SG}$. (i) Let $u' \in (N_G(u) \cap \texttoptiebar{C}_{SG}) \cap X \ne \emptyset$ and $v' \in (N_G(v) \cap \texttoptiebar{C}_{SG}) \cap X \ne \emptyset$. Then by Claim \ref{closed path}, $u'\texttoptiebar{C}_{SG}v'$ and $u'uvv'$ form a closed $X$-3-path, a contradiction. 
(ii) Let $(N_G(u) \cap \texttoptiebar{C}_{SG}) \cap X \ne \emptyset$ and $(N_G(v) \cap \texttoptiebar{C}_{SG}) \cap X = \emptyset$. It is assumed that $N_G(u) \cap \texttoptiebar{C}_{SG} \subseteq X$, for otherwise $u \in \texttoptiebar{C}_{SG}$. By Claim \ref{closed path}, for $p, q \in N_G(v) \cap \texttoptiebar{C}_{SG}$, $p$ and $q$ are in a closed $X$-3-path, say $R$. Let $p' \in (N_G(p) \cap \texttoptiebar{C}_{SG}) \cap X$, $p'' \in (N_G(p) \cap \texttoptiebar{C}_{SG}) \cap (V \setminus X)$, $q' \in (N_G(q) \cap \texttoptiebar{C}_{SG}) \cap X$ and $q'' \in (N_G(q) \cap \texttoptiebar{C}_{SG}) \cap (V \setminus X)$. If $R$ is composed of $pp'Rq'q$ and $pp''Rq''q$ then $pp''Rq''q$ and $pvq$ forms a closed $X$-3-path, a contradiction. Thus, $R$ is composed of $pp'Rq''q$ and $pp''Rq'q$. If $p'$ and $q'$ form a $X$-3-path $S$ between them distinct from $R$ then $q'Rp''p$, $pvq$, $qq''Rp'$ and $p'Sq'$ form a closed $X$-3-path, a contradiction. By Claim \ref{closed path}, it suffices that $N_G(p') \setminus R \subseteq X$. Since we assume that $C_{SG}$ is unique, for $X \cup \{v\} = X'$, $X'$ is a d-set of $G[V(\texttoptiebar{C}_{SG}) \cup R]$. Since a graph has a d-set $Y$ such that $E(Y)$ is minimal, if $E(X')$ is not minimal then $C_{SG}$ is not unique. Since $N_G(p') \setminus R \subseteq X'$, $E(X')$ is not minimal, a contradiction. 
(iii) Let $(N_G(u) \cap \texttoptiebar{C}_{SG}) \cap X = \emptyset$ and $(N_G(v) \cap \texttoptiebar{C}_{SG}) \cap X = \emptyset$. By the same argument of (ii), a contradiction follows. Let $R = \{u\}$. (iv) Let $p \in (N_G(u) \cap \texttoptiebar{C}_{SG}) \cap X \ne \emptyset$ and $q \in (N_G(u) \cap \texttoptiebar{C}_{SG}) \cap (V \setminus X) \ne \emptyset$. Let $q' \in (N_G(q) \cap \texttoptiebar{C}_{SG}) \cap X$ and $q'' \in (N_G(q) \cap \texttoptiebar{C}_{SG}) \cap (V \setminus X)$. Then by Claim \ref{closed path}, $p\texttoptiebar{C}_{SG}q'q$ and $puq$ form a closed $X$-3-path, a contradiction. (v) Let $N_G(u) \cap \texttoptiebar{C}_{SG} \subseteq V \setminus X$. By the same argument of (ii), a contradiction follows. Thus, if $|R| = 1$ then $x \in R$ satisfies $N_G(x) \subseteq X$.
\end{proof}

By Claim \ref{R1}, a d-set $X$ of $G$ is obtained by assigning $X$-3-paths to $D_{SG}$. Because if $C \in C_{SG}$ does not have a $X$-3-path then $C \cap X = \emptyset$. Let $C_{SG}$ be not unique. Let $\mathcal{D}$ be a set of $\bigoplus \texttoptiebar{D}_{SG}$. For $D_1 \in \mathcal{D}$, if $V(G) = D_1$ then $D_1$ is a desired one. Otherwise, a set of $V(G) \setminus \bigoplus \texttoptiebar{D}_{SG}$ is pairwise disjoint so that $|\mathcal{D}| \leq |V(G)|$. This completes the proof of Theorem \ref{T2}. \qed

\paragraph{Note}
The abstract of this paper appeared in 15th Cologne-Twente Workshop on Graphs and Combinatorial Optimization.

\end{document}